\documentclass[a4paper,12pt]{article}

\usepackage[utf8]{inputenc}
\usepackage[T1]{fontenc}
\usepackage{fixltx2e}
\usepackage{graphicx}
\usepackage{longtable}
\usepackage{float}
\usepackage{wrapfig}
\usepackage{multirow}
\usepackage{rotating}
\usepackage[normalem]{ulem}
\usepackage{amsmath}
\usepackage{amsthm}
\usepackage{textcomp}
\usepackage{marvosym}
\usepackage{bm}
\usepackage{wasysym}
\usepackage{amssymb}
\usepackage{capt-of}
\usepackage{hyperref}
\usepackage{comment}
\usepackage{color}
\usepackage{subfigure}

\def\sizesmallfig{0.4}
\def\sizewave{2}

\def\Sb{\mathbf{S}}
\def\Pb{\mathbf{P}}
\def\A{\mathbf{A}}
\def\B{\mathbf{B}}

\def\q{\mathbf{q}}

\def\U{\mathbf{U}}

\def\Z{\mathbf{Z}}

\def\D{\mathbf{D}}

\def\X{\mathbf{X}}

\def\x{\mathbf{x}}

\def\y{\mathbf{y}}

\def\z{\mathbf{z}}
\renewcommand\u{\mathbf{u}}

\def\M{\mathbf{M}}
\def\Q{\mathbf{Q}}
\def\Lop{\mathcal{L}}

\newcommand{\Sigmab}{\boldsymbol{\Sigma}}
\newcommand{\R}{\mathbb{R}}
\newcommand{\C}{\mathbf{C}}
\newcommand{\N}{\mathbb{N}}

\theoremstyle{plain}
\newtheorem{theorem}{Theorem}[section]
\newtheorem{lemma}[theorem]{Lemma}

\theoremstyle{definition}
\newtheorem{definition}[theorem]{Definition}

\newtheorem{assumption}[theorem]{Assumption}

\newtheorem{remark}{Remark}

\if{
\newcommand{\qed}{\nobreak \ifvmode \relax \else
      \ifdim\lastskip<1.5em \hskip-\lastskip
      \hskip1.5em plus0em minus0.5em \fi \nobreak
      \vrule height0.75em width0.5em depth0.25em\fi}
}\fi

\begin{document}
\author{Victor Magron$^{1}$ \and Christophe Prieur$^{2}$}
\date{\today}
\title{Optimal Control of PDEs using Occupation Measures and SDP Relaxations}

\footnotetext[1]{Univ.~Grenoble Alpes, CNRS, VERIMAG, 700 av Centrale, 38401 Saint-Martin d'Hères, France (\url{victor.magron@univ-grenoble-alpes.fr})}
\footnotetext[2]{Univ.~Grenoble Alpes, CNRS, GIPSA-lab, 38401 Saint-Martin d'Hères, France (\url{christophe.prieur@gipsa-lab.fr})}
\maketitle


\begin{abstract}
This paper addresses the problem of solving a class of nonlinear optimal control problems (OCP) with infinite-dimensional linear state constraints involving Riesz-spectral operators. 
Each instance within this class has time/control dependent polynomial Lagrangian cost and control contraints described by polynomials. 
We first perform a state-mode discretization of the Riesz-spectral operator. 
Then, we approximate the resulting finite-dimensional OCPs by using a previously known hierarchy of semidefinite relaxations. Under certain compactness assumptions, we provide a converging hierarchy of semidefinite programming relaxations whose optimal values yield lower bounds for the initial OCP. 
We illustrate our method by two numerical examples,  involving a diffusion partial differential equation and a wave equation. We also report on the related experiments. 
\end{abstract}
\section{Introduction}
\label{sec:intro}
This paper focuses on solving a general class of nonlinear optimal control problems (OCP) with infinite-dimensional linear state constraints. The states of such systems are assumed to be solution of the following abstract Cauchy problem on a Hilbert space $\Z$:
\begin{align}
\label{eq:cauchy}
\dot{\z} = A \, \z(t) + B \, u(t) \,, \quad \z(0) = \z^0 \,,
\end{align}
where  $A$ is the infinitesimal generator of a strongly continuous semigroup on $\Z$, the control space $\U$ being a Hilbert space  and $B$ being an operator mapping $\U$ into $\Z$.

The main underlying motivation is the study of several parabolic/hyperbolic partial differential equations (PDEs), which can be formulated into such infinite-dimensional linear systems (see e.g.~\cite{Curtain95,MorrisHB}). This includes diffusion equations to model heat flow in rods, beam equations for vibration analysis, wave equations to describe deflection of vibrating strings~\cite{meurer2012control}. Another motivation is the control or stability analysis of delay systems. Such systems involve retarded differential equations and can also be reformulated as abstract differential equations involving Riesz-spectral operators. For more details, we refer the interested reader to~\cite[\S~2.4]{Curtain95}.

There exists already a literature on optimal control of PDEs. Consider in particular the Hilbert Uniqueness Method (HUM) as introduced in~\cite{lions1988exact} which gives the optimal control by minimizing the $L^2$-norm of the control. A numerical method dedicated to the computation of the optimal HUM control for the wave equation is given in~\cite{lebeau2010experimental}. Numerically computing the optimal asks for special care since the high frequencies may interfere with the mesh (see~\cite{zuazua2005propagation} for a general result of convergence in the parabolic case). With respect to these prior works, the aim of this paper is to present a new method for the computation of optimal controls for a large class of linear partial equations, in presence of state and control constraints, and using moments and measure theory.

Our general methodology is inspired from prior research work in the context of dynamical polynomial systems. In~\cite{Las08}, the authors proposed a framework to handle certain class of nonlinear OCPs while relying on infinite-dimensional linear programming (LP), where variables are occupation measures. They derived a hierarchy of semidefinite programs (SDP), whose optimal values form a nondecreasing sequence of lower bounds for the LP, potentially converging to the optimal value of the OCP when all data are polynomial. This can be seen as leverage on the technique initially introduced in~\cite{Lasserre01moments} in the context of static polynomial optimization. Recent extensions allowed to address various problems arising in control of polynomial systems, including the characterization of regions of attraction~\cite{HK14roa}, maximum controlled invariants~\cite{KHJ13mci} as well as reachability analysis of discrete-time polynomial systems~\cite{MGHT17reach}.

The approach developed in~\cite{Valmorbida:IEEETAC:2016} is completely different with respect to the present one. Indeed for a class of polynomials PDEs, a design method is suggested in~\cite{Valmorbida:IEEETAC:2016} to compute a polynomial Lyapunov function and an associate domain of attraction (see also~\cite{ahmadi2017safety} in a safety certification context). To do so, techniques relying on sums of squares are applied and a fully different hierarchy of SDP relaxations is derived. In our paper, the control problem is different, since a stability analysis is not considered but rather an optimal control problem. Moreover, in our case occupation measures are used on linear Riesz-spectral operators.

\paragraph{Contribution}

In this article, we consider the class of nonlinear OCPs with (time-control dependent) polynomial Lagrangian cost, infinite-dimensional linear state constraints and semialgebraic control constraints. By contrast with prior work, our main results can be summarized as follows:
\begin{itemize}
\item we rely on a discretization of the Riesz-spectral operator involved in the state constraints to provide a sequence of nonlinear OCPs with finitely many linear state-mode  constraints. Each state-mode OCP can then be reformulated as an infinite-dimensional LP by using occupation measures, following the approach described in~\cite{Las08}. At the end, this leads to a sequence of infinite-dimensional LPs, indexed by the number of modes.
\item we approximate each LP with a hierarchy of finite-dimensional SDP relaxations whose optimal values form a converging sequence of lower bounds for the LP and thus of the corresponding state-mode OCP as well as of the initial OCP. This methodology is then applied to study nonlinear control problems involving diffusion, beam or wave partial differential equations. 
\end{itemize}

Our paper is organized as follows. In Section~\ref{sec:background}, we recall preliminary background about Riesz-spectral operators and state our initial OCP. Section~\ref{sec:infLP} is devoted to the discretization of this OCP into a sequence of state-mode OCPs, which are in turn reformulated into infinite-linear LPs over occupation measures. Next, in Section~\ref{sec:sdp}, we show how to build a hierarchy of converging finite-dimensional SDP relaxations for each infinite-dimensional LP. Eventually, we illustrate the method with several numerical experiments in Section~\ref{sec:bench}.
\section{Riesz-Spectral Operators and OCPs}
\label{sec:background} 
\subsection{Riesz-Spectral Operators}
\label{sec:riesz}
Let  $\Z$ be a Hilbert space, equipped with a norm operator denoted by $\|\cdot\|$. Let $L(\Z)$ be the set of bounded linear operators from $\Z$ to $\Z$.
%
As stated in~\cite[Definition~2.1.2]{Curtain95}, a \textit{strongly continuous semigroup} ($\mathcal{C}^0$-\textit{semigroup}) is an operator-valued function $T(t)$ from $\R^+$ to $L(\Z)$ that satisfies the following properties:
\begin{align}
T(t+s) & = T(t) \, T(s) \,, \quad \forall t,s \geq 0 \,, \\
T(0) & = I \,, \quad (\text{identity on~} \Z)  \,, \\
\| T(t) \, \z^0 - \z^0 \| & \to 0 \,, \quad \text{as } t \to 0^+ \,, \quad \forall \z^0 \in \Z \,.
\end{align}
%
The \textit{infinitesimal generator} A of a $\mathcal{C}^0$-\textit{semigroup} on a Hilbert space $\Z$ is defined by:
\begin{align*}
A z = \lim_{t \to 0^+}  (T(t) - I) \z \,,
\end{align*}
whenever the limit exists (see~\cite[Definition~2.1.8]{Curtain95}). The domain $\D(A)$ of $A$ is the set of elements in $\Z$ for which the limit exists.
%

The notion of \textit{Riesz-spectral operator} allows to represent several interesting classes of linear partial differential equations, which can be either parabolic or hyperbolic. 
This representation relies on the concept of \textit{Riesz basis}. 
\begin{definition}{(\cite[Definition~2.3.1]{Curtain95})}
A sequence of vectors $(\Phi_k)_{k\geq 1}$ in  $\Z$ is a \textit{Riesz basis} for $\Z$ if the following two conditions hold:
\begin{enumerate}
\item $\overline{\mathop{\text{span}}\limits_{k\geq 1}} \{\Phi_k \} = \Z$.
\item There exists $m, M > 0$ such that for all positive $N \in \N$ and arbitrary scalars $\alpha_1,\dots,\alpha_N$ one has:
\[
m \sum_{k=1}^N |\alpha_k|^2 \leq 
\| \sum_{k=1}^N \alpha_k \Phi_k \|^2 \leq 
M \sum_{k=1}^N |\alpha_k|^2 \,.
\]
\end{enumerate}
\end{definition}
Let $A$ be a linear closed operator on $\Z$ with simple eigenvalues $(\lambda_k)_{k \geq 1}$  and corresponding eigenvectors $(\Phi_k)_{k \geq 1}$ forming a Riesz basis in $\Z$. If the closure of $(\lambda_k)_{k \geq 1}$ is disconnected, then $A$ is called a \textit{Riesz-spectral operator}.

In the sequel, we assume that $A$ is such a  Riesz-spectral operator.  We note $(\Psi_k)_{k \geq 1}$ the eigenvectors of $A^\star$ satisfying $\langle \Phi_k, \Psi_j \rangle = \delta_{k j}$. Then by~\cite[Theorem~2.3.5]{Curtain95}, $A$ has the following representation, for all $\z \in \D(A)$:
\begin{align}
\label{eq:representation}
A \, \z = \sum_{k=1}^\infty \lambda_k \langle \z, \Psi_k \rangle \Phi_k \,,
\end{align}
\subsection{Statement of Optimal Control Problem}
\label{sec:pb}
For a given Hilbert space $\Z$, real number $T > 0$ and a positive integer $m$, we consider the following infinite-dimensional {\em state linear system}:
\begin{align}
\label{eq:ocp}
\dot{\z} = A \, \z(t) + B \, u(t) \,, \quad t \in [0, T] \,,
\end{align}
where 
\begin{itemize}
\item $A$ is a Riesz-spectral operator, with representation as in~\eqref{eq:representation} and is the infinitesimal generator of a strongly continuous semigroup $T(t)$ on the Hilbert space $\Z$.
\item $B$ is a bounded linear operator from $\U$ to $\Z$, with $\U$ being a compact subset of $\R^m$. The control function $u : [0, T] \to \U$ is  bounded measurable.
\end{itemize}
Let $\Z^T$ be a subset of $\D(A)$.
For a given $T > 0$, and $\z^0 \in \D(A)$, a control $u$ is said admissible on $[0, T]$ if there exists a solution $\z(\cdot)$ of 
\eqref{eq:ocp}, such that $\z(0) = \z^0$, $u$ is well defined on $[0, T]$ and fulfills:
\begin{align}
\label{eq:state}
(\z(t),u(t)) \in \Z \times \U \quad \text{a.e.~on~}  [0, T] \,,
\end{align}
\begin{align}
\label{eq:bounds}
\z(T) \in \Z^T \,.
\end{align}
We denote by $\mathcal{U}^T$ the set of admissible controls on $[0, T]$. For a given $\u \in \mathcal{U}^T$, the cost of the associated trajectory $\z(\cdot)$ is given by:
\begin{align}
\label{eq:cost}
J(0,T,\z^0, \u) = \int_0^T h(t,u(t)) d t \,,
\end{align}
with $h:[0, \infty] \times \U \to \R$ being a smooth function.

In the sequel, we consider the optimal control problem of computing a trajectory solution $\z$ of~\eqref{eq:ocp} starting from $\z^0$, fulfilling the state and control constraints~\eqref{eq:state}, the terminal constraint~\eqref{eq:bounds} and minimizing the cost~\eqref{eq:cost}.

For fixed final time $T$, we define:
\begin{align}
\label{eq:JT}
J(0, T, \z^0) := \inf_{\u \in \mathcal{U}^T} J(0,T,\z^0, \u) \,,
\end{align}
and for free final time, we define:
\begin{align}
\label{eq:Jfree}
J(0, \z^0) := \inf_{T > 0, \u \in \mathcal{U}^T} J(0,T,\z^0, \u) \,.
\end{align}
\section{Infinite-Dimensional LP and Occupation Measures}
\label{sec:infLP}
\subsection{State-Mode Discretization}
\label{sec:disc}
For a state linear system defined as in~\eqref{eq:ocp}, let us define the $N$-th state-mode $z_N := \langle \z, \Psi_N \rangle$, for all positive integers $N$. Similarly, we define $z_N^0 := \langle \z^0, \Psi_N \rangle$.
Let us note $B u = \sum_{i=1}^m B_i u_i$ 
with $B_i : \R \to \Z$ being the $i$-th projection of the operator $B$ in the canonical basis of $\R^m$. 
We define the linear operator $b_N(u) := \sum_{i=1}^m \langle B_i u_i, \Psi_N \rangle$. Next, let us note $f_N(t,z_N,u) := \lambda_N \, z_N + b_N ( u(t))$.
The state linear system defined as in~\eqref{eq:state} is then equivalent to the following infinite-dimensional state-mode linear system:
\begin{align}
\label{eq:ocpN}
\dot{z}_N = f_N(t,z_N,u) \,, \quad t \in [0, T] \,, \quad N \in \{1,2,\dots\} \,.
\end{align}
\begin{remark}
In the sequel, we assume without loss of generality that the eigenvalues $(\lambda_N)_{N \geq 1}$ involved in~\eqref{eq:ocpN} are real numbers. Indeed, if $(\lambda_N)_{N \geq 1}$ is a sequence of complex numbers, one can always rewrite~\eqref{eq:ocpN} as a real system of linear equations involving the real and imaginary parts of $z_N$ and $\lambda_N$, thus  yielding a system with twice the initial number of variables. 
\end{remark}

We note $\Z_N := \{\langle \z, \Psi_N \rangle : \z \in \Z \}$ and $\Z_N^T := \{\langle \z, \Psi_N \rangle : \z \in \Z^T \}$. Accordingly to~\cite{Las08}, we also use the notations $\Sigmab_N := [0, T] \times \Z_1 \times \dots \times \Z_N$ and $\Sb_N := \Sigmab_N \times U$.

Instead of considering OCP~\eqref{eq:JT}, we consider the optimal control problem of minimizing the cost~\eqref{eq:cost} over the set $\mathcal{U}_N^T$ of admissible controls on $[0, T]$ such that there exists a trajectory solution $(z_1,\dots,z_N)$ of~\eqref{eq:ocpN} starting from $ (z_1^0,\dots,z_N^0) = (z_1(0),\dots,z_N(0))$, and fulfilling the constraints:
\begin{align}
\label{eq:stateN}
(z_1(t),\dots,z_N(t), u(t)) \in \Z_1 \times \dots \times \Z_N \times \U \quad \text{a.e.~on~}  [0, T] \,,
\end{align}
\begin{align}
\label{eq:boundsN}
\z(T) \in \Z_1^T \times \dots \times \Z_N^T \,.
\end{align}
Thus, for fixed final time $T$, we define:
\begin{align}
\label{eq:JNT}
J_N(0, T, \z^0) := \inf_{\u \in \mathcal{U}_N^T} J(0,T,\z^0, \u) \,,
\end{align}
and for free final time, we define:
\begin{align}
\label{eq:JNfree}
J_N(0, \z^0) := \inf_{T > 0, \u \in \mathcal{U}_N^T} J(0,T,\z^0, \u) \,.
\end{align}

\begin{lemma}
\label{th:JNTleqJT}
Assume that the set $\mathcal{U}^T$ of admissible controls is non empty. Then the sequence $(J_N(0,T,\z^0))_{N>0}$ is monotone  nondecreasing and $\lim_{N \to \infty} J_N(0, T, \z^0) \leq J(0, T, \z^0)$.
\end{lemma}
\begin{proof}
Let us fix a positive integer $N$ and prove that $\mathcal{U}_N^T \subseteq \mathcal{U}^T$. 
Let us consider a control $\u \in \mathcal{U}^T$ together with a  solution $z(\cdot)$ of~\eqref{eq:ocp} starting from $\z^0$, fulfilling the state and control constraints~\eqref{eq:state}, the terminal constraint~\eqref{eq:bounds}. Thus, the same $\u$ together with the N first modes $z_1,\dots,z_N$ must satsify the state and control constraints~\eqref{eq:stateN}, as well as the terminal constraint~\eqref{eq:boundsN}. This proves that  $J_N(0, T, \z^0) \leq J(0, T, \z^0)$. Similarly, one has $\mathcal{U}_{N+1}^T \subseteq \mathcal{U}_N^T$ which yields $J_N(0, T, \z^0) \leq J_{N+1}(0, T, \z^0)$. The sequence $(J_N(0,T,\z^0))_{N>0}$ is monotone nondecreasing and bounded, which implies the desired convergence result.
\end{proof}
Note that OCP~\eqref{eq:JNT} is a particular case of OCP~(2.5) in~\cite{Las08} with the notations $x \leftarrow (z_1,\dots, z_N)$, $f \leftarrow (f_1, \dots, f_N)$. In our case, the Lagrangian $h$ involved in the cost does not depend on the state variables $z_1,\dots, z_N$ and the Hamiltonian $H$ is equal to 0.

In the sequel, we make the following assumption:
\begin{assumption}
\label{hyp:compact}
For all positive integers $N$, the set $\Z_N$ (resp.~$\Z_N^T$) is a closed real interval $[\underline{z_N}, \overline{z_N}]$ (resp.~$[\underline{z_N}^T, \overline{z_N}^T]$).
\end{assumption}
This assumption allows to use the results from~\cite{Las08} relying on compactness of the general (resp.~terminal) state constraints. In particular, one shows the equivalence between  OCP~\eqref{eq:JNT} and an infinite-dimensional LP over occupation measures, for all positive integer $N$.
\subsection{Occupation Measures}
\label{sec:meas}
Let us define the linear mapping $\Lop_N : \mathcal{C}^1(\Sb_N) \to \mathcal{C}(\Sigmab_N)$:
\begin{align}
\label{eq:LopN}
v \mapsto \Lop_N (g) :=\frac{\partial g}{\partial t}(t,z_1,\dots,z_N) + \sum_{k=1}^N f_k(t,z_k,u) \, \frac{\partial g}{\partial z_k} (t,z_1,\dots,z_N) \,,
\end{align}
Then, we define further the adjoint linear operator $\Lop_N^\star : \mathcal{C}(\Sigmab_N)^\star \to \mathcal{C}^1(\Sb_N)^\star$ with the adjoint relation:
\begin{align}
\label{eq:LopNstar}
\langle \Lop_N^\star (\mu), g \rangle := \langle \mu, \Lop_N (g) \rangle = \int_{\Sb_N} \Lop_N (g) \, d \mu(t,z_1,\dots,z_N,u) \,,
\end{align}
for each Borel measure $\mu \in \mathcal{M}(\Sb_N) = \mathcal{C}(\Sb_N)^\star$ and any test function $g \in \mathcal{C}^1(\Sigmab_N)$.
We note $\mathcal{T}$ (resp.~$\mathcal{B}_N$) the Borel $\sigma$-algebra associated with $[0, T]$ (resp.~$\R^N$).
Let $\mathbf{1}_\A (\cdot)$ stand for the indicator function of the set $\A$, namely $\mathbf{1}_\A(\z) = 1$ if $\z \in \A$ and $\mathbf{1}_\A(\z) = 0$ otherwise.
Let us now define the two following {\em occupation measures} $\mu_N$ and $\mu_N^T$, by
\begin{align}
\label{eq:occupationN}
\mu_N (\A \times \B \times \C) := & \int_{[0, T]\cap \A} \mathbf{1}_{\B \times \C}[z_1(t),\dots,z_N(t),u(t)] d t \,, \\ 
\mu_N^T (\D) := & \mathbf{1}_{\D}(z_1^T,\dots,z_N^T) \,,
\end{align}
%
%
for each $\D \in \mathcal{B}_N$ and for all rectangles $(\A \times \B \times \C)$ with $(\A,\B,\C) \in \mathcal{T} \times \mathcal{B}_N \times \mathcal{B}_m$.

The {\em average occupation measure} $\mu_N$ counts the average time spent by the $N$ first state-modes and control trajectories in subsets of $\Z_1 \times \dots \times \Z_N \times \U$.
The {\em terminal occupation measure} $\mu_N^T$ counts the average time spent by the $N$ first state-modes in subsets of $\Z_1^T \times \dots \times \Z_N^T$.

Let us note $\otimes$ the product of measures, $\delta_N^0$ the Dirac measure at point $(t,z_1,\dots,z_N) = (0, z_1^0,\dots,z_N^0)$ and $\delta^T_N$ the Dirac measure at point $(t,z_1,\dots,z_N) = (T, z_1^T,\dots,z_N^T)$

As in~\cite[$\S$~2.3]{Las08}, we obtain the following linear equation between the initial measure $\delta^0_N$, the terminal measure $\mu_N^T$ and the average occupation measure $\mu_N$:
\begin{align}
\label{eq:liouvilleN}
\delta_N^T \otimes \mu_N^T - \delta^0_N = \Lop_N^\star (\mu_N) \,,
\end{align}
\subsection{Infinite-Dimensional LP Formulation}
\label{sec:inflp}
Let us consider the following infinite-dimensional linear program (LP) indexed by the number of modes $N$:
\begin{equation}
\label{eq:lpN}
\begin{aligned}
\Pb_N := \inf\limits_{\mu_N, \mu_N^T} \quad  &  \langle h, \mu_N \rangle \\		
\text{s.t.} 
\quad &  \delta_N^T \otimes \mu_N^T - \delta^0_N = \Lop_{N}^\star (\mu_N)  \,,\\
\quad & \mu_N \in  \mathcal{M}_+([0, T]  \times \Z_1 \times \dots \Z_N \times \U)\,, \\
\quad & \mu_N^T \in  \mathcal{M}_+(\Z_1^T \times \dots \Z_N^T)\,.
\end{aligned}
\end{equation}
The following theorem is a consequence of~\cite[Theorem~2.3]{Las08} in our context.
\begin{theorem}
\label{th:cvgLP}
Assume that the set $\mathcal{U}^T$ of admissible controls is non empty.
Then,
\begin{itemize}
\item[(i)] For all positive $N \in \N$, LP~\eqref{eq:lpN} is solvable, that is, the inf is attained and $\Pb_N \leq J_N (0, T, \z^0) \leq J (0, T, \z^0)$.
\item[(ii)] For all positive $N \in \N$, there is no duality gap between LP~\eqref{eq:lpN} and its dual. 
\item[(iii)] If for every $(t,\z) \in [0, T] \times \Z$, the set $\{ A \, \z + B \, u : u \in \U \}$ is convex and the function 
\[
\begin{aligned}
v \mapsto g_{t,\z}(v) := \inf\limits_{u \in \U} \quad  &  h (t,u) \\
\text{s.t.}
\quad &  v = A \, \z + B \, u \,,\\
\end{aligned}
\]
is convex, then for all positive $N \in \N$, OCP~\eqref{eq:JNT} has an optimal solution and $\Pb_N = J_N (0, T, \z^0)$.
\end{itemize}
\end{theorem}
\begin{proof}
The set $\mathcal{U}^T$ is non empty. Thus, for all positive $N \in \N$, the set $\mathcal{U}_N^T$ is also non empty.
As shown in~\cite[\S~2.3]{Las08}, LP~\eqref{eq:lpN} is feasible whenever there exists an admissible control for OCP~\eqref{eq:JNT}. The right inequality in (i) follows from Lemma~\ref{th:JNTleqJT}.
Using Assumption~\ref{hyp:compact}, the proofs of the left inequality in (i) and (ii) follow respectively from~\cite[Theorem~2.3~(i)]{Las08} and~\cite[Theorem~2.3~(ii)]{Las08}.

Let us note $f := (f_1,\dots, f_N)$.
To prove (iii), we need to show that the two assumptions of~\cite[Theorem~2.3~(iii)]{Las08} are ensured, namely that:
\begin{enumerate}
\item For every $(t, z_1,\dots,z_N) \in \Sigmab_N$, the set $f(t,z_1,\dots,z_N,\U)$ is convex.
\item For every $(t, z_1,\dots,z_N) \in \Sigmab_N$, the function 
\[
\begin{aligned}
v \mapsto \inf\limits_{u \in \U} \quad  &  h (t,u) \\		
\text{s.t.} 
\quad &  v = f(t,z_1,\dots,z_N,u) \,,\\
\end{aligned}
\]
is convex.
\end{enumerate}
Let us fix $(z_1,\dots,z_N) \in \Sigmab_N$ and complete this vector with infinitely many zeros to obtain an element $(t,\z) \in [0, T] \times \Z$. Then the first (resp.~second) item follows from the convexity of $\{ A \, \z + B \, u : u \in \U \}$ (resp.~$g_{t,\z}$), yielding the desired result.
\end{proof}
\section{A Hierarchy of SDP Relaxations}
\label{sec:sdp}
Here, we apply the results from~\cite[\S~3.3]{Las08} to derive SDP relaxations of the infinite-dimensional LP~\eqref{eq:lpN}, yielding a converging sequence of lower bounds for $\Pb_N$. This also allows to obtain lower bounds for the cost $J_N (0, T, \z^0)$ of OCP~\eqref{eq:JNT} as well as the cost $J(0, T, \z^0)$ of the initial problem, i.e.~OCP~\eqref{eq:JT}. Section~\ref{sec:mom} is dedicated to preliminary background on moment matrices. Then, we provide the hierarchy of finite-dimensional SDP relaxations of $\Pb_N$ in Section~\ref{sec:sdpocp}.
\subsection{Moment Matrices}
\label{sec:mom}
%
%
We assume that $\U$ is a  basic compact semialgebraic subset of $\R^m$. Using Assumption~\ref{hyp:compact}, this gives the following definitions for all positive integer $N$:
\begin{align} 
\label{eq:defU}
\U & :=  \{u \in \R^m : w_1(u) \geq 0, \dots, w_s(u) \geq 0 \} \,,\\
\label{eq:defZ}
\Z_1 \times \dots \times \Z_N & :=  \{(z_1,\dots,z_N) \in \R^N : v_1(z_1) \geq 0,\dots, v_N(z_N) \geq 0 \} \,, \\
\label{eq:defZT}
\Z_1^T \times \dots \times \Z_N^T & :=  \{(z_1,\dots,z_N) \in \R^N : v_1^T(z_1) \geq 0,\dots, v_N^T(z_N) \geq 0 \} \,,
\end{align}
for polynomials $w_1,\dots,w_s \in \R[u]$, $v_N(z) := (\overline{z_N} - z) (z - \underline{z_N})$ and $v_N^T(z) := (\overline{z_N}^T - z) (z^T - \underline{z_N}^T)$, where $\underline{z_N}$ and $\overline{z_N}$ are as in Assumption~\ref{hyp:compact}.

We set $r_j := \lceil (\deg g_j ) / 2 \rceil, j = 1, \dots, s$. For the ease of further notation, we set $g_0(\x) := 1$. For sequel purpose (Section~\ref{sec:sdpocp}), we also define $r_{\min} := \max \{1, r_1, \dots,r_m \}$.

\if{
\paragraph{Sums of squares} Let $\Sigma[u]$ stand for the cone of polynomial sums of squares (SOS) and let us note $\Sigma_r[u]$ the cone of SOS polynomials of degree at most $2 r$, namely $\Sigma_r[u] := \Sigma[u] \cap \R_{2r}[u]$.

For each $r \geq r_{\min} := \max \{1, r_1, \dots,r_m \}$, let $\Q_r$ be the $r$-truncated quadratic module generated by $g_0, \dots, g_m$:
\begin{align*}
\Q_r & := \Bigl\{\,\sum_{j=0}^{m} \sigma_j(u) {g_j} (u) : \sigma_j \in \Sigma_{r - r_j}[u], \,  j = 0, \dots, s  \,\Bigr\} \,.
\end{align*}
}\fi
To guarantee the convergence behavior of the relaxations presented in the sequel, we need to ensure that polynomials which are positive on $\U$ lie admit certain representations. The existence of such representations is guaranteed by Putinar's Positivstellensaz (see e.g.~\cite[Section 2.5]{lasserre2009moments}), when the following condition holds:
\begin{assumption}(Archimedean property)
\label{hyp:archimedean}
There exists a large enough integer $M$ such that one of the polynomials $w_j$ describing the set $\U$ is equal to $M - \| u \|_2^2$.
\end{assumption}
\if{
In addition, the semialgebraic set $\X$ is assumed to be ``simple'' (e.g.~a ball or a box), meaning that $\X$ fulfills the following condition: 
\begin{assumption}
\label{hyp:momb}
The moments of the Lebesgue measure on $\X$ are available analytically.
\end{assumption}
}\fi
From now on we assume that the control set $\U$ is a basic compact semialgebraic set as in~\eqref{eq:defU} and that it satisfies the Archimdean property stated in Assumption~\eqref{hyp:archimedean}. Note that the definition of each polynomial $v_N$(resp.~$v_N^T$) involved in~\eqref{eq:defZ} (resp.~\eqref{eq:defZT}) ensures that this Archimedean property also holds for all state sets.

\paragraph{Moment matrices}
Given a positive $N\in\N$, a multi-index $\alpha$ is a vector of $N$ nonnegative integers $\alpha := (\alpha_1, \dots, \alpha_N)$.
For all $r \in \N$, we set $\N^{N}_r := \{ \alpha \in \N^{N} : \sum_{i=1}^{N} \alpha_i \leq r \}$, whose cardinality is $\binom{N+r}{r}$. We define $\N^{m}_r$ in a similar way.

Then a polynomial $\theta \in \R[t,z_1,\dots,z_N,u]$ is written as follows:
\[(t,z_1,\dots,z_N,u) \mapsto \theta(t,z_1,\dots,z_N,u) \,=\,\sum_{p \in \N,\alpha \in \N^N, \beta \in\N^m} \, \theta_{p \alpha \beta} \,  t^p z_1^{\alpha_1} \dots z_N^{\alpha_N}  u^\beta \:, \]
and $\theta$ is identified with its vector of coefficients $\theta=(\theta_{\gamma})$ in the canonical basis of monomials indexed by $\gamma \in \N \times \N^N \times \N^m = \N^{1+N+m}$.

Given a real sequence $\y =(y_{\gamma})_{\gamma \in \N^{1+N+m}}$, let us define the linear functional $\ell_\y : \R[t,z_1,\dots,z_N,u] \to \R$ by $\ell_\y(\theta) := \sum_{\gamma} \theta_{\gamma} y_{\gamma}$, for every polynomial $\theta$. 
Then, we associate to $\y$ the so-called {\it moment} matrix $\M_r(\y)$, that is the real symmetric matrix  with rows and columns indexed by $\N_r^{1+N+m}$ and the following entrywise definition: 
\[ 
(\M_r(\y))_{\beta,\gamma} := y_{\beta+\gamma}   \,, \quad
\forall \beta, \gamma \in \N_r^{1+N+m} \,. 
\]

Now, we recall the following important fact allowing to build the hierarchy of semidefinite relaxations in Section~\ref{sec:sdpocp}.
If $\y$ is the sequence of moments of a nonnegative measure $\mu$ supported on $\R \times \R^N \times \R^m$, then one has for every polynomial $q \in \R[t,z_1,\dots,z_N,u]$, with vector of coefficients denoted by $\q$:
\[
\q^T \, \M_r(\y)) \, \q = \ell_\y(q^2) = \int_{\R \times \R^N \times \R^m} q^2 \, d \mu \geq 0 \,.
\]
Thus, a necessary condition for $\mu$ to be  nonnegative with support in $\R \times \R^N \times \R^m$ is that $\M_r(\y)$ is a semidefinite positive matrix, for all $r$ (which is denoted by $\M_r(\y) \succeq 0$ ).

Given a polynomial $\theta \in \R[t,z_1,\dots,z_N,u]$, we also associate to $\y$ 
the so-called {\it localizing} matrix, that is the real symmetric matrix $\M_r(\theta \, \y)$ with rows and columns indexed by $\N_r^{1+N+m}$ and the following entrywise definition: 
\[ 
(\M_r(\theta \, \y))_{\beta, \gamma} := \sum_{\delta} \theta_\delta y_{\delta+\beta+\gamma} 
\,, \quad
\forall \beta, \gamma \in \N_r^{1+N+m} \,. 
\]
Let $\y$ be the moment sequence of a nonnegative measure $\mu$ supported in the super level set of $\theta$.
As for the moment matrix, a necessary condition  is that $\M_r(\theta \, \y) \succeq 0$, for all $r$.
\subsection{A Hierarchy of Finite-Dimensional SDP Relaxations}
\label{sec:sdpocp}
We note $g_T(z_1,\dots,z_N) := g(z_1,\dots,z_N)$ for any $g \in \mathcal{C}^1(\Sigmab_N)$. For each positive $N\in\N$, given a sequence $\y = (y_\gamma)$ indexed in the monomial basis of $\R[t,z_1,\dots,z_N]$, let us note $\y(t)$, $\y(z_1),\dots,\y(z_N)$ and $\y(u)$ its marginals with respect to the variables $t$, $z_1,\dots,z_N$ and $u$ respectively. These marginal sequences are indexed in the monomial basis of $\R[t]$, $\R[z_1],\dots,\R[z_N]$ and $\R[u]$ respectively so that the following holds:
\[
\y(t) = (z_{p,0,0})_{p \in \N} \,, \quad 
\y(z_i) = (z_{0,k e_i,0})_{k \in \N} \,, \quad 
\y(u) = (z_{0,0,\beta})_{\beta \in \N^m} \,, 
\]
where $e_i$ is the $N$-th dimensional vector whose only nonzero entry is 1 at index $i$.

For all $r \geq r_{\min}$, let us consider the following SDP relaxations of LP~\eqref{eq:lpN}:
\begin{equation}
\label{eq:sdprelax}
\begin{aligned}
\Pb_{N,r} := \inf\limits_{\y, \y^T} \quad & \ell_\y (h) \\	
\text{s.t.} 
\quad & \M_{r} (\y), \M_{r} (\y^T) \succeq 0 \,, \quad \M_{r-1} (t (T - t) \y(t))\succeq 0 \,, \\
\quad & \M_{r - 1}(v_i \, \y(z_i)) \succeq 0 \,, 
\quad i = 1,\dots, N \,,\\
\quad & \M_{r - r_j}(g_j \, \y(u)) \succeq 0 \,, 
\quad j = 1,\dots, s \,,\\
\quad & \M_{r - 1}(v_i^T \, \y^T) \succeq 0 \,, 
\quad i = 1,\dots, N \,,\\
\quad & \ell_{\y^T}(g_T) -  \ell_{\y} (\Lop_N(v)) = v(0,z_1^0,\dots,z_N^0) \,, \quad \forall g = (t^p z_1^{\alpha_1} \dots z_N^{\alpha_N})  \\
\quad & \text{with } p + |\alpha| \leq r \,.
\end{aligned}
\end{equation}
Let us briefly recall why SDP~\eqref{eq:sdprelax} is a relaxation of LP~\eqref{eq:lpN}. The reason is that for any couple of Borel measures $(\mu_N,\mu_N^T)$  feasible for LP~\eqref{eq:lpN}, then the corresponding couple of associated moment sequences $(\y,\y^T)$ is feasible for SDP~\eqref{eq:sdprelax} 
(we removed the subindex $N$ in the notation of moment sequences for conciseness).

The following theorem states that this hierarchy of SDP relaxations provides a sequence of lower bounds converging to the optimal value of LP~\eqref{eq:lpN}. This theorem can be seen as a  reformulation of~\cite[Theorem~3.6]{Las08} in our context.
\begin{theorem}
\label{th:sdpcvg}
Let $\U$, $\Z_1 \times \dots \times \Z_N$ and $\Z_1^T \times \dots \times \Z_N^T$ be basic compact semialgebraic sets defined respectively in~\eqref{eq:defU},~\eqref{eq:defZ} and~\eqref{eq:defZT} and suppose that Assumption~\ref{hyp:archimedean} holds. Then, for all positive $N \in \N$,
\begin{itemize}
\item[(i)] $\Pb_{N,r} \uparrow \Pb_{N}$ as $r \to \infty$.
\item[(ii)] In addition, if for every $(t,\z) \in [0, T] \times \Z$, the set $\{ A \, \z + B \, u : u \in \U \}$ is convex and the function 
\[
\begin{aligned}
v \mapsto g_{t,\z}(v) := \inf\limits_{u \in \U} \quad  &  h (t,u) \\
\text{s.t.}
\quad &  v = A \, \z + B \, u \,,\\
\end{aligned}
\]
is convex, then $\Pb_{N,r} \uparrow \Pb_{N} = J_N (0, T, \z^0) \leq J(0, T, \z^0)$ as $r \to \infty$. 
\end{itemize}
\end{theorem}
\begin{proof}
See the proof of~\cite[Theorem~3.6]{Las08}. 
The second inequality of (ii) comes from Theorem~\ref{th:JNTleqJT}. 
\end{proof}
For the case of OCP~\eqref{eq:JNfree} with free terminal time, $T$ becomes a variable lying in a given interval $[0, T_0]$ and we need to solve another SDP relaxation. This relaxation (omitted for the sake of conciseness) is obtained as in~\cite[\S~3.3]{Las08} and involves the moments of the occupation measure $\y^T$, which is now supported on $[0, T_0] \times \Z_1^T \times \Z_N^T$. 
\section{Numerical Experiments}
\label{sec:bench}
Here, we present experimental benchmarks that illustrate our method. 
\subsection{Software Implementation}
\label{sec:soft}
For a given positive $N \in \N$, an instance of OCP~\eqref{eq:JNT}, we build the SDP~\eqref{eq:sdprelax} at relaxation order $r$. The overall algorithmic scheme is implemented by using the {\sc POCP} package~\footnote{\url{http://homepages.laas.fr/henrion/software/pocp}} available in {\sc Matlab}.
{\sc POCP} allows to solve (nonlinear) OCPs with polynomial data. Here, we use {\sc POCP} to relax OCP~\eqref{eq:JNT} as SDP~\eqref{eq:sdprelax},  the relaxation process being handled by {GloptiPoly 3}~\cite{gloptipoly} via {\sc Yalmip}~\cite{yalmip} toolbox.
We compute the solution of SDP~\eqref{eq:sdprelax} thanks to the interface between {\sc GloptiPoly} and the SDP solver {\sc SeDuMi 1.3}~\cite{sedumi}.

All benchmarks were performed with a PC Intel Quad-Core CPU (3.30GHz) with 16Gb of RAM, running under Debian 8.
\subsection{Diffusion Equation}
\label{sec:heat}
We consider the example of the one-dimensional heat diffusion model in the Hilbert space $\Z = L^2([0, 1])$ of square integrable functions with control in $\U = [-1, 1]$. This model is described as follows by the linear PDE with boundary conditions, given by:
\begin{align}
\label{ex:heat}
\dfrac{\partial h}{\partial t} & = \dfrac{\partial^2 h}{\partial x^2} + b(x) \, u(t) \,, \quad \forall t \in [0, T] \,, \nonumber \\
\dfrac{\partial h}{\partial x}(0,t) & = \dfrac{\partial h}{\partial x}(1,t) = 0 \,, \quad \forall t \in [0, T] \,,  \\ 
h(x, 0) & = \cos (\pi x) \,, \quad  h(x, T) = 0 \,, \forall x \in [0, 1] \,, \nonumber 
\end{align}
with $b(x) := \dfrac{1}{\epsilon} \mathbf{1}_{[x_0 - \epsilon, x_0 + \epsilon]}(x)$.
This heat diffusion PDE can be described by using an abstract Cauchy system as in~\eqref{eq:cauchy}, with $A \, h := \dfrac{\partial^2 h}{\partial x^2}$ defined on the domain:
\[
\D(A) = \{ 
h \in \Z : \dfrac{\partial h}{\partial x}, \dfrac{\partial^2 h}{\partial x^2} \in \Z \text{ and } h'(0) = h'(1) = 0
\}
\]
and $B : \U \to \Z$ being the bounded linear operator defined by $B \, u = b \, u$. As shown e.g.~in~\cite{Curtain95}, the operator $A$ is a Riesz-spectral operator with eigenvalues given by $\lambda_N = -N^2 \pi^2$ and corresponding eigenvectors given by $\Phi_N(x) = \sqrt{2} \sin (N \pi x)$, for all positive integer $N$.

Now we describe our results after solving three successive SDP relaxations of the minimal time problem, i.e.~OCP~\eqref{eq:JNfree} with $h = 1$, $N = 3$, $T_0 = 1$, $\epsilon = 0.4$ and $x_0 = 0.27$, for increasing values of the relaxation order (from $r = 2$ to $r = 6$). The lower bounds of the minimal time are $\Pb_{3,2} = 0.203$, $\Pb_{3,4} = 0.302$ and $\Pb_{3,6} = 0.359$, all being computed in less than 5 seconds. After solving the SDP relaxations, we perform the controller extraction by following the procedure described in~\cite[\S~5]{Korda14}, whose output $u_{N,r}$ is a polynomial approximation of degree $r/2$.
\begin{figure}[!ht]
\centering
\subfigure[$r = 2$]{
\includegraphics[scale=\sizesmallfig]{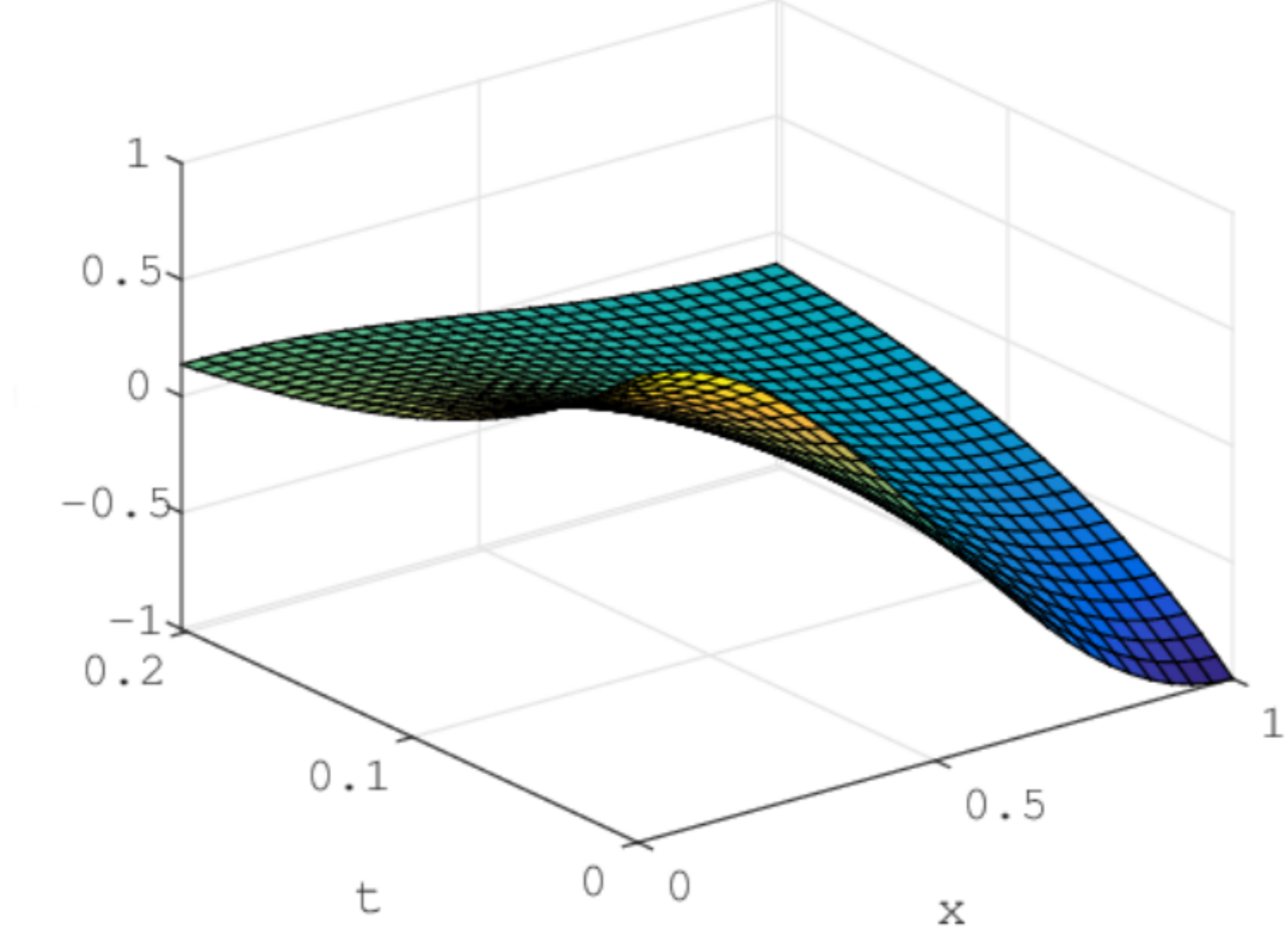}}
\subfigure[$r = 4$]{
\includegraphics[scale=\sizesmallfig]{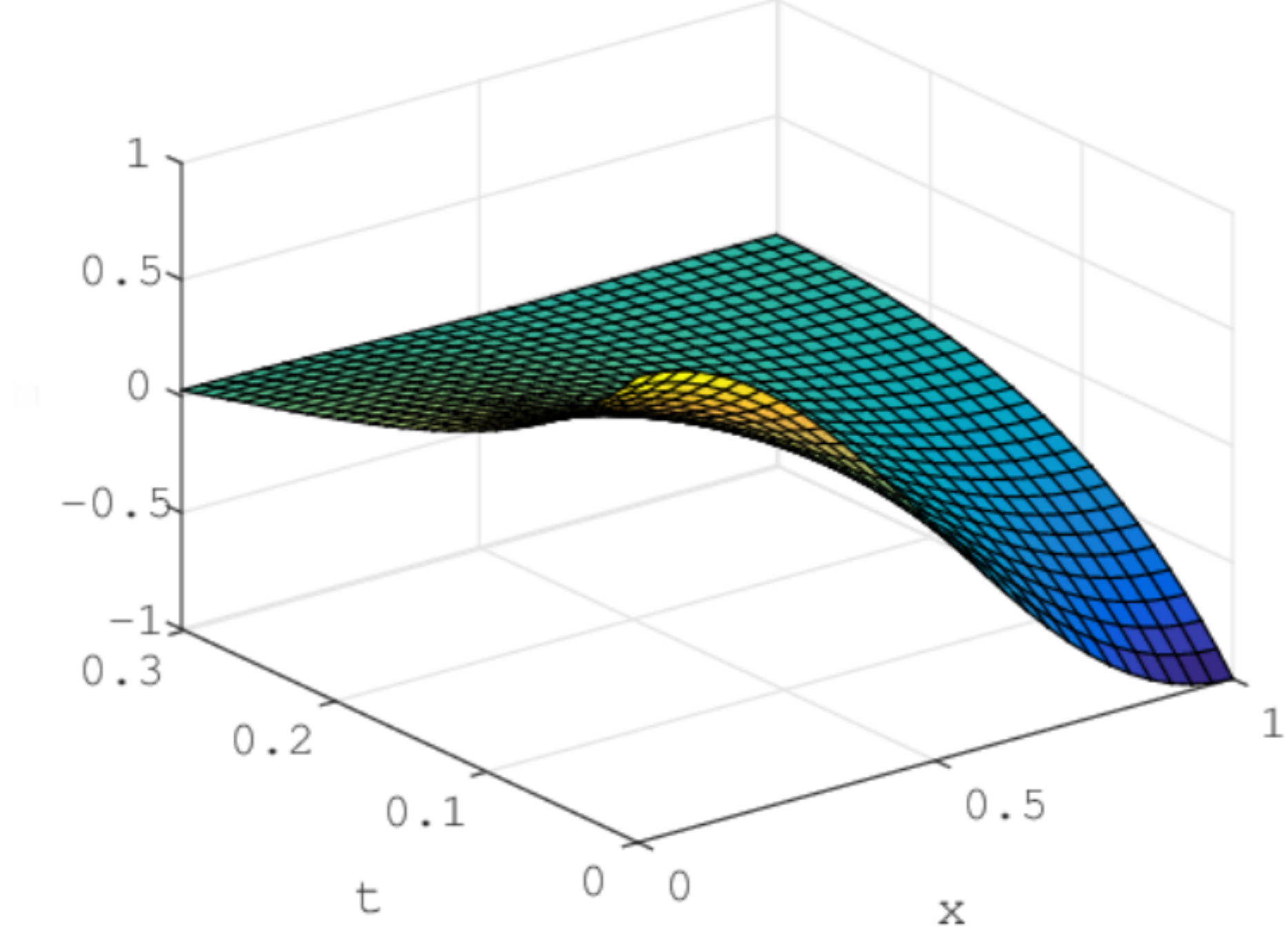}}
\\
\subfigure[$r = 6$]{
\includegraphics[scale=\sizesmallfig]{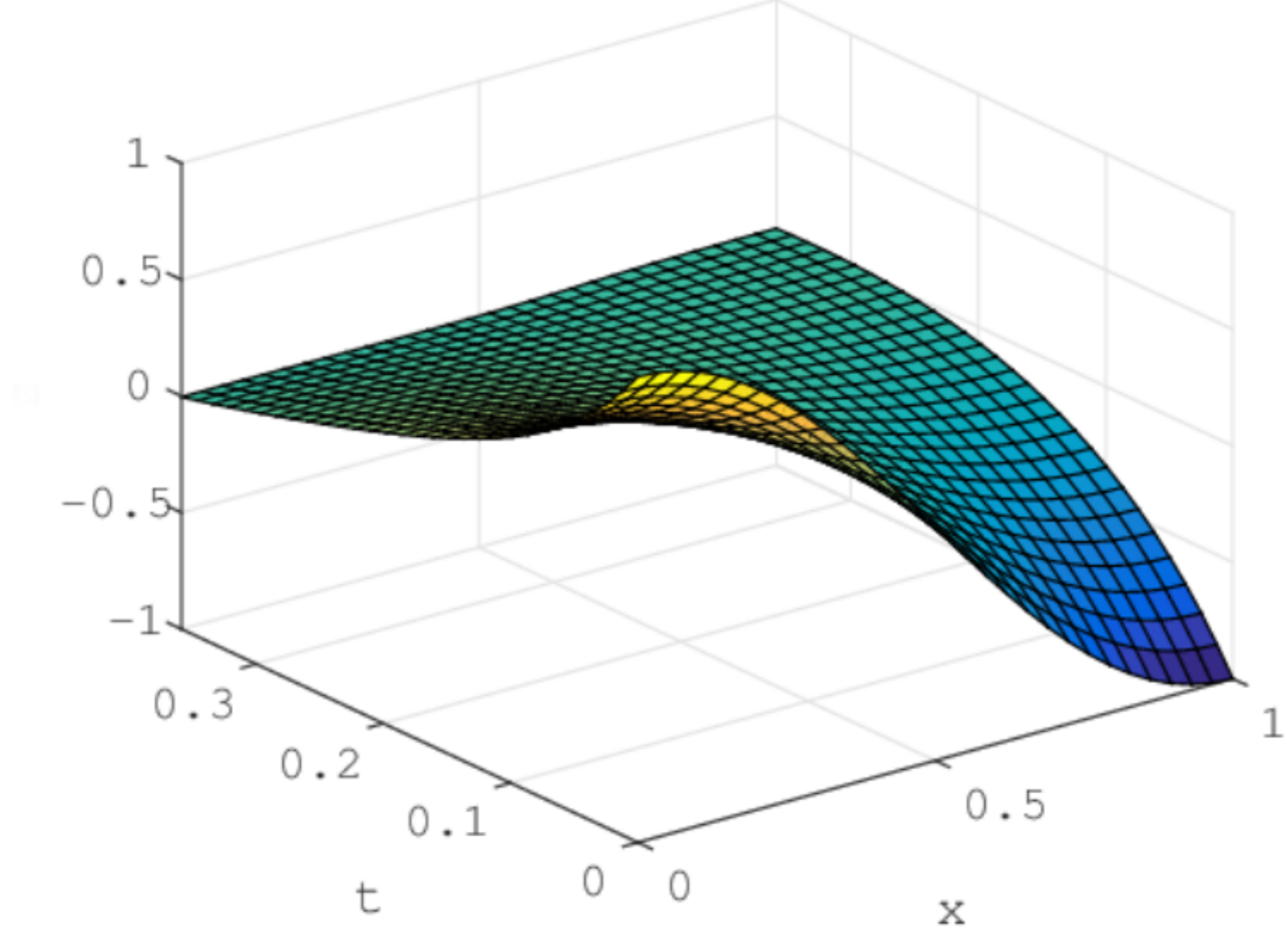}}
\caption{Approximate trajectories for minimal time OCP involving PDE~\eqref{ex:heat}}
\label{fig:heat}
\end{figure}

Then, we compute a numerical approximation of the solution $h$ of PDE~\eqref{ex:heat} with $u = u_{N,r}$, while relying on the~\texttt{pdepe} procedure available inside {\sc Matlab}. The size of the spatial mesh (resp.~number of time steps) is 30. The corresponding results are displayed on Figure~\ref{fig:heat}. The left picture (corresponding to $r = 2$) shows that the lower bound obtained by SDP is too coarse as the value of $h$ does not satisfy the final condition $h(x, T) = 0$ at time $T = \Pb_{3,2} = 0.203$. Increasing the value of the relaxation order allows to obtain controllers ensuring that the solutions to~\eqref{ex:heat} vanish at the end of the simulation.
\subsection{Wave Equation}
\label{sec:wave}
Next, we consider the one-dimensional model of wave equation with $\Z$ and $\U$ as in Section~\ref{sec:heat}, described the linear PDE with boundary conditions given by:
\begin{align}
\label{ex:wave}
\dfrac{\partial^2 w}{\partial t} & = \dfrac{\partial^2 w}{\partial x^2} + b(x) \, u(t) \,, \quad \forall t \in [0, T] \,, \nonumber \\
w(x,0) & = \sin (2 \pi x) \,, \quad \forall x \in [0, 1] \,, \\
\dfrac{\partial w}{\partial t}(x,0) & = \sin^2 (2 \pi x) \,, \quad \forall x \in [0, 1] \,, \nonumber \\
w(0,t) & = w(1,t) = 0 \,, \quad \forall t \in [0, T] \,, \nonumber
\end{align}
with $b(x)$ as in Section~\ref{sec:heat}. The wave equation can be described by using an abstract Cauchy system associated to a Riesz-spectral operator with eigenvalues $\lambda_{N} = j N \pi$ for all nonzero $N \in \Z$ and a corresponding Riesz basis of eigenfunctions $\Phi_{N} = \dfrac{1}{j N \pi} (\sin (N \pi x),  \  j N \pi \sin (N \pi x))^T$. Here $j$ denotes the unit imaginary number with $j^2 = 1$.

To solve the minimal time problem OCP~\eqref{eq:JNfree}, we first reformulate~\eqref{eq:ocpN} as a real system by considering the real and imaginary parts of $z_N$ and $\lambda_N$. Note that when fixing the number of modes $N$, the OCP has dimension $4 N$. Thus, one ends up computing the cost value $\Pb_{N,r}$ of SDP~\eqref{eq:sdprelax} which involves$\binom{4 N + 2 r}{ 2 r}$ variables at relaxation order $r$. This comes together with a computational burden, as reported in Table~\ref{table:wave} for values of $N$ and $r$ between 1 and 4. For each $(N,r)$, ``vars'' stands for the total number of SDP variables considered when solving~\eqref{eq:sdprelax} with SeDuMi. The occurrence of the symbol ``$-$'' means that the SDP problem could not be solved within one day of computation.
\begin{table}[!ht]
\begin{center}
\caption{Timing results to solve SDP~\eqref{eq:sdprelax} for PDE~\eqref{ex:wave}}
\begin{tabular}{c|c|cccc}
\hline
\multicolumn{2}{c|}{relaxation order $r$}
 & 1 & 2 & 3 & 4
\\
\hline            
\multirow{2}{*}{$N=1$} & vars &  $42$ & $279$ & $1134$ & $3498$ \\
& time (s) & $1.25 $ & $3.70 $ &  $17.5$ & $156 $ \\
\hline
\multirow{2}{*}{$N=2$} & vars &  $100$ & $1496$ & $11011$ & $-$ \\
& time (s) & $2.58 $ & $40.1$ &  $4518$ & $-$ \\
\hline
\multirow{2}{*}{$N=3$} & vars &  $211$ & $4880$ & $-$ & $-$ \\
& time (s) & $5.30$ & $496.$ &  $-$ & $-$ \\
\hline
\multirow{2}{*}{$N=4$} & vars &  $342$ & $12160$ & $-$ & $-$ \\
& time (s) & $9.16$ & $5824$ &  $-$ & $-$ \\
\hline
\end{tabular}
\label{table:wave}
\end{center}
\end{table}

As in Section~\ref{sec:heat}, we obtain polynomial approximations $u_{N,r}$ of the controller and represent them on Figure~\ref{fig:wave} for $N=1,2,3$. At the first mode ($N=1$), one notices that the approximations obtained at higher relaxation orders ($r=3$ and $r=4$) are close to each other, which seems to indicate a convergence behavior. For higher number of modes, such behavior is hard to investigate due the quickly growing size of SDP problems.
\begin{figure}[!ht]
\centering
\subfigure[$N = 1$, $r \in \{1,2,3,4\}$]{
\includegraphics[scale=\sizewave]{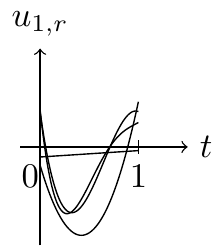}}
\subfigure[$N = 2$, $r \in \{1,2,3\}$]{
\includegraphics[scale=\sizewave]{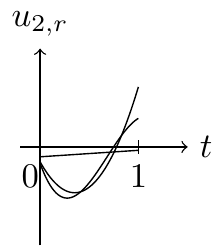}}
\subfigure[$N = 3$, $r \in \{1,2\}$]{
\includegraphics[scale=\sizewave]{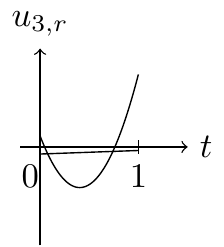}}
\caption{Approximate control $u_{N,r}$ for minimal time OCP~\eqref{eq:JNfree} related to PDE~\eqref{ex:wave}}
\label{fig:wave}
\end{figure}

\section{Conclusion}
\label{sec:end}
This paper presented a general approximation framework to handle certain classes of OCPs with infinite-dimensional linear state constraints involving Riesz-spectral operators. The first step consists of a state-mode discretization of the operator to obtain an optimization problem over occupation measures. The second step consists of solving SDP relaxations involving finite numbers of moments of the occupation measures. 
By selecting only a few modes and number of moments, preliminary experiments demonstrate that our method is practically able to control systems involving PDEs. A first research direction would be to investigate sparse/symmetric systems involving higher number of modes.
Another topic of interest would be to develop extensions to several classes of nonlinear PDEs, including the case when the dynamics (resp.~cost) are polynomial functions of the state.
\bibliographystyle{plain}

\end{document}